\documentclass[a4paper,reqno,final,11pt]{amsart}
\usepackage{amssymb}
\usepackage{amsmath}
\usepackage{amsfonts}
\usepackage{hyperref}
\usepackage{graphicx}

\newtheorem{theorem}{Theorem}
\theoremstyle{plain}
\newtheorem{corollary}{Corollary}

\newtheorem{remark}{Remark}
\numberwithin{equation}{section}

\email{amsengouga@gmail.com,abdelmouhcene.sengouga@univ-msila.dz}
\subjclass[2010]{35L05, 35C10, 93D15.}
\keywords{Axially travelling strings, Fourier series, energy estimates, boundary stabilization.
}

\begin{document}
\title[Decay estimates for an axially travelling string]{Energy decay estimates for an axially travelling string damped at one end}

\author[S. Ghenimi]{Seyf Eddine Ghenimi}
\author[A. Sengouga]{Abdelmouhcene Sengouga}
\address[Seyf Eddine Ghenimi, Abdelmouhcene Sengouga]{ Laboratory of Functional Analysis and Geometry of Spaces\\
Department of mathematics\\
Faculty of Mathematics and Computer Sciences\\
University of M'sila\\
28000 M'sila, Algeria.}
\date{\today}
\maketitle

\begin{abstract}
We study the small vibrations of an axially travelling string with a dashpoint damping at one end. The string is modelled by a wave equation in a time-dependent interval with two endpoints moving at a constant speed $v$. For the undamped case, we obtain a conserved functional equivalent to the energy of the solution. We derive precise upper and lower estimates for the exponential decay of the energy with explicit constants. These estimates do not seem to be reported in the literature even for the non-travelling case $v=0$.\end{abstract}

%%%%%%%%%%%%%%%%%%%%%%%%%%%%%%%%%%%%%%%%%%%%%%%%%%%%%%%%%%%%%%%%%%%%%%%%%%%%%%%%%%%%%%%%%%%%%%%%%%%%%%%%%%%%%%%%

\section{Introduction}

We consider small transversal vibrations of a uniform string travelling with
a constant speed $v$ between to two pulleys (inlet and outlet) kept at a
fixed distance $L$. The mechanical setting is sketched in Figure \ref{fig1}
where the inlet is fixed while the outlet is allowed to move transversely
and attached to a damping device (a dashpoint with a damping factor $\eta $).

\begin{figure}[tbph]
\centering\includegraphics[width=0.67\textwidth]{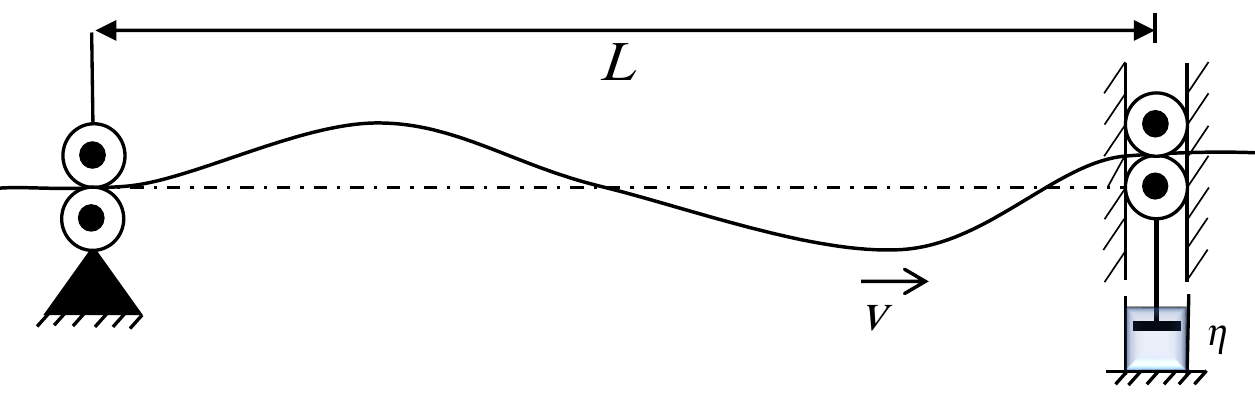}
\caption{An axially travelling string with a dash-point at the outlet pulley.}
\label{fig1}
\end{figure}

Many mechanical devices with axially moving continua, such as power
transmission chains and belts, magnetic tapes, band saws and fibre winders,
see for instance \cite{MaKa2014,HoPh2019,Chen2005,Mira1960,Mote1965}, are
limited in their efficiency and utility due to unwanted vibrations. As a
result, stabilization of axially moving systems is necessary to reduce or
eliminate these vibrations and improve the overall performance and
productivity of these mechanical systems.

The existing approaches in the literature describe the above problem in
fixed space coordinates, see for instance \cite%
{Chen2005,MaKa2014,TaYi1997,ChFe2014,CHZW2019,CYFZ2021,LeMo1996}. Gaiko and
van Horssen \cite{GaHo2015} considered a simplified mathematical model
describing the small vibration of the string with a mass-spring-dashpoint
damping at the outlet. Under the restriction that the speed $v$ and the damping
factor are small, i.e. $0<v\ll 1$ and $\eta \ll 1$ the authors in \cite%
{GaHo2015} obtained an asymptotic approximation for the solution of (\ref%
{wave0L}) using a multiple scale approach.

Denoting the displacement function by $u$, depending
on the position $s$ along the string and the time $\tau $, with only a
dashpoint damping at the outlet, their model can be stated as follows 
\begin{equation}
\left\{ 
\begin{array}{ll}
u_{\tau \tau }+2vu_{s\tau }-\left( 1-v^{2}\right) u_{ss}=0, & \text{ for}%
\medskip \ s\in \left( 0,L\right) \text{ and } \tau >0, \\ 
u\left( 0,\tau \right) =0,\text{ \ \ } & \text{ for}\medskip \text{ }\tau >0,
\\ 
\left( 1-v^{2}\right) u_{s}\left( L,\tau \right) +\left( \eta -v\right)
u_{\tau }\left( L,\tau \right) =0,\text{ \ } & \text{ for}\medskip \text{ }%
\tau >0, \\ 
u(s,0)=u^{0}\left( s\right) ,\text{ \ \ }u_{\tau }\left( s,0\right)
=u^{1}\left( s\right) , & \text{ for }s\in \left( 0,L\right) ,%
\end{array}%
\right.  \label{wave0L}
\end{equation}%
where the subscripts $\tau $ and $s$ stand for the derivatives in time and
space variables respectively. The functions $u^{0}$ and $u^{1}$ represents
the initial shape and the initial transverse speed of the string,
respectively. 

In the present work, we do not consider the magnitudes of $v$ and $\eta $ as
small ones. We only assume that $\eta \geq 0$ and that speed $v$ is strictly
less then the speed of propagation of the wave (here normalized to $c=1$),
i.e.%
\begin{equation}
0\leq v<1.  \label{tlike}
\end{equation}%
If the speed $v$ approaches the critical speed $c=1$, an instability will
occur, as shown by \cite{Mote1965,LeMo1996}. Another key difference with
most of the existing works is that we consider the model in a moving space
coordinates. We introduce the variables%
\begin{equation*}
s=L-x+vt\text{ \ and \ }\tau =t,
\end{equation*}%
hence 
\begin{equation*}
x\in \mathbf{I}_{t}:=\left( vt,L+vt\right) ,\text{ for }t\geq 0,
\end{equation*}%
which is an interval travelling in the positive sense of the real axis (as
in \cite{GhSe2022,Mira1960}). It follows that%
\begin{equation*}
\partial _{s}=-\partial _{x}\text{ \ \ and \ \ }\partial _{\tau }=v\partial
_{x}+\partial _{t}.
\end{equation*}%
Rewriting Problem (\ref{wave0L}) in the new coordinates, we obtain the
following (pure) wave equation with a damping at the moving boundary $x=vt,$%
\begin{equation}
\left\{ 
\begin{array}{ll}
\phi _{tt}-\phi _{xx}=0,\ \smallskip & \text{for }x\in \mathbf{I}_{t}\text{
and }t>0, \\ 
\left( 1-\eta v\right) \phi _{x}\left( vt,t\right) -\left( \eta -v\right)
\phi _{t}\left( vt,t\right) =0,\text{\ }\smallskip \text{ } & \text{for }t>0,
\\ 
\phi \left( L+vt,t\right) =0\text{,}\smallskip & \text{for }t>0, \\ 
\phi (x,0)=\phi ^{0}\left( x\right) ,\text{ }\phi _{t}\left( x,0\right)
=\phi ^{1}\left( x\right) ,\text{ \ } & \text{for }x\in \mathbf{I}_{0},%
\end{array}%
\right.  \tag{WP}  \label{waveq}
\end{equation}%
where $\phi ^{0}=u^{0}$ and $\phi ^{1}=u^{1}-v\phi _{x}^{0}.$

Let us denote%
\begin{equation*}
\gamma _{v}:=\frac{1+v}{1-v},\text{\ \ }\gamma _{\eta }:=\frac{1+\eta }{%
1-\eta }\text{ \ and \ }\omega _{n}:=\left\{ 
\begin{array}{ll}
\displaystyle\frac{\left( 2n+1\right) }{2}i\pi -\frac{1}{2}\ln \gamma _{\eta
}, & \text{if }0\leq \eta <1,\smallskip  \\ 
\displaystyle ni\pi -\frac{1}{2}\ln \left\vert \gamma _{\eta }\right\vert ,
& \text{if }\eta >1.%
\end{array}%
\right. 
\end{equation*}%
Note that
\begin{itemize}
	\item $\ \gamma _{v}\geq 1$ for $0\leq v<1$,
	\item $\left\vert \gamma _{\eta
}\right\vert \geq 1$ and the real part of $\omega _{n}$
remains negative, for $\eta >0$ with $\eta \neq 1$.
\end{itemize}

As a first result, we derive a closed form of the solution of (\ref{waveq})
given by the series formulas%
\begin{equation}
\phi (x,t)=\sum_{n\in \mathbb{Z}}a_{n}\left( \gamma _{\eta }e^{\frac{1-v}{L}%
\omega _{n}(t+x)}+e^{\frac{1+v}{L}\omega _{n}(t-x)}\right) ,\text{ \ \ \ for 
}x\in \mathbf{I}_{t}\text{ and }t\geq 0,  \label{exact0}
\end{equation}%
which is a sum of two waves travelling in opposite directions. The
coefficients $a_{n}$ are explicitly given in function of the initial data $%
\phi ^{0}$ and $\phi ^{1},$ see Theorem \ref{thexist1}.

Next, we demonstrate how the series formulas (\ref{exact0}) can be used to
achieve the following results:

\begin{itemize}
\item \emph{For the undamped case, i.e. $\eta =0$}, we show that the
functional\footnote{%
Here and in the sequel, the subscript $v$ is used to emphasize the
dependence on the speed $v$.}%
\begin{equation}
\mathcal{E}_{v}\left( t\right) =\frac{1}{2}\int_{\mathbf{I}_{t}}\left( \phi
_{t}+v\phi _{x}\right) ^{2}+\left( 1-v^{2}\right) \phi _{x}^{2}dx,\text{ \ \
\ for }t\geq 0,  \label{QE}
\end{equation}%
depending on $L,t,v$ and the solution $\phi $ of (\ref{waveq}) is conserved
in time, see Theorem \ref{thcv}. Note that under the assumption (\ref{tlike}%
), the functional $\mathcal{E}_{v}$ is positive-definite and we will call it
the "energy" of the solution $\phi $.

\item \emph{For the damped case $\eta >0$ with $\eta \neq 1$}, the (usual)
energy%
\begin{equation}
E_{v}\left( t\right) =\frac{1}{2}\int_{\mathbf{I}_{t}}\phi _{t}^{2}\left(
x,t\right) +\phi _{x}^{2}\left( x,t\right) dx,\ \ \ \text{for }t\geq 0,
\label{E}
\end{equation}%
depending on $L,t,v$ and the solution of (\ref{waveq}) decays exponentially.
More precisely%
\begin{equation}
\frac{1}{\gamma _{\eta }^{2}\gamma _{v}}E_{v}\left( 0\right) e^{-\frac{%
1-v^{2}}{L}\ln \left\vert \gamma _{\eta }\right\vert t}\leq E_{v}\left(
t\right) \leq \gamma _{\eta }^{2}\gamma _{v}E_{v}\left( 0\right) e^{-\frac{%
1-v^{2}}{L}\ln \left\vert \gamma _{\eta }\right\vert t},\ \ \ \text{for }%
t\geq 0.  \label{exp-decay}
\end{equation}%
See Theorem \ref{th-stabl} and it corollaries for more sharper estimates. It follows in
particular that there is no decay to zero in finite time.
\end{itemize}

The exponential decay estimate given in (\ref{exp-decay}), with the sharp
rate and explicit constants, is new to the best to our knowledge. The
approach presented here relays on Fourier series and Parseval's identity, and
it does not involve semigroup theory as in \cite{LeMo1996,FuWW1999}. Even for a
non-travelling string, i.e. when $v=0,$ the precise estimate 
\begin{equation}
\frac{1}{\gamma _{\eta }^{2}}E_{0}\left( 0\right) e^{-\frac{1}{L}\ln
\left\vert \gamma _{\eta }\right\vert t}\leq E_{0}\left( t\right) \leq
\gamma _{\eta }^{2}E_{0}\left( 0\right) e^{-\frac{1}{L}\ln \left\vert \gamma
_{\eta }\right\vert t},\ \ \ \text{for }t\geq 0,  \label{estv=0}
\end{equation}%
seems to be not reported in the literature. The existing results assure only
that 
\begin{equation*}
E_{0}\left( t\right) \leq K E_{0}\left( 0\right) e^{-\frac{1}{L}\ln \left\vert \gamma _{\eta
}\right\vert t},
\end{equation*}
for a (non-explicitly given) positive constant $K$, see for instance \cite%
{Vese1988,CoZu1995,Cher1994,QuRu1977}.

For the special case $\eta =1$ and $0\leq v<1,$ the boundary condition at $%
x=vt$ reads 
\begin{equation*}
\phi _{x}\left( vt,t\right) -\phi _{t}\left( vt,t\right) =0.
\end{equation*}%
This is a transparent condition, i.e. there is no reflections of waves from
the boundary $x=vt$ and consequently all the initial disturbances leave the
interval $\left( vt,L+vt\right) $ at most after a time 
\begin{equation*}
T_{v}:=\frac{L}{1+v}+\frac{L}{1-v}=\frac{2L}{1-v^{2}}.
\end{equation*}
That is to say that the linear velocity feedback $\phi _{t}\left( vt,t\right) $ steers
the solution to the zero state in the finite time $T_{v}$. See for instance 
\cite{LeMo1996} for the case $0<v<1$ and \cite{Vese1988,CoZu1995} for $v=0$.

After the present introduction, we derive the exact solution and the\
expression for the coefficients of the series formula (\ref{exact0}). In
Section 3, we show that the energy $\mathcal{E}_{v}$ of the undamped
equation is constant in time. In the last section, we show the exponential
stability when $\eta >0$.

\section{Exact solution}

To compute the solution of Problem (\ref{waveq}), given by (\ref{exact0}),
we need to compute the coefficients $a_{n},n\in 
%TCIMACRO{\U{2124} }%
%BeginExpansion
\mathbb{Z}
%EndExpansion
.$ To this end, we need to know the functions $\phi ^{0}$ and $\phi ^{1}$ on
an interval larger than $\mathbf{I}_{0}=\left( 0,L\right) $. As in \cite%
{Seng2020,GhSe2022}, we introduce%
\begin{equation*}
L_{2}:=\frac{2L}{1-v}
\end{equation*}%
and extend $\phi $ to the interval $\left( L+vt,L_{2}+vt\right) $ by setting%
\begin{equation}
\tilde{\phi}(x,t)=\left\{ 
\begin{array}{ll}
\phi \left( x,t\right) , & x\in \left( vt,L+vt\right) ,\smallskip \\ 
-\phi \left( \frac{1}{\gamma _{v}}\left( vt-x\right) +\frac{2L}{1+v}%
+vt,t\right) ,\text{ \ \ } & x\in \left( L+vt,L_{2}+vt\right) .%
\end{array}%
\right.  \label{phi0+}
\end{equation}%
The obtained function is well defined since the first variable of $\phi $
remains in the interval $\left( vt,L+vt\right) $. In particular, the
homogeneous boundary condition $\phi \left( L+vt,t\right) =0$ remains
satisfied, for every $t\geq 0$.

\begin{remark}
Clearly, $0<L\leq L_{2}/2\ $for $0\leq v<1.$ If $v=0,$ then $L_{2}=2L$ and
the function $\tilde{\phi}$ on $\left( L,2L\right) $ is an odd function on $%
\left( 0,2L\right) $, with respect to $x=L$. If $0<v<1,$ then $\phi $ is
extended as an odd function with an extra dilatation on the added interval $%
\left( L+vt,L_{2}+vt\right) $, see Figure \ref{fig2}. 
\begin{figure}[tbph]
\centering\includegraphics[width=0.67\textwidth]{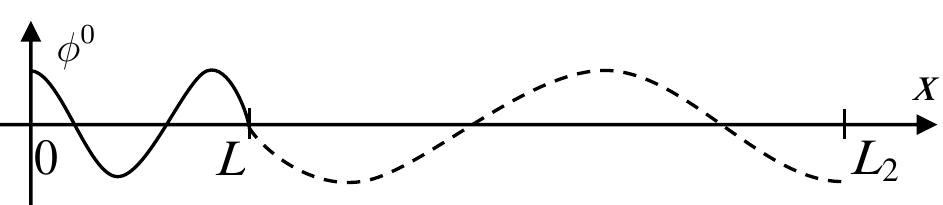}
\caption{Example of the extension of an initial data $\protect\phi ^{0}$
when $0<v<1$.}
\label{fig2}
\end{figure}
\end{remark}

Taking the derivative of (\ref{phi0+}) with respect to $x$, we obtain%
\begin{equation}
\tilde{\phi}_{x}(x,t)=\left\{ 
\begin{array}{ll}
\phi _{x}\left( x,t\right) , & x\in \left( vt,L+vt\right) ,\smallskip \\ 
\frac{1}{\gamma _{v}}\phi _{x}\left( \frac{1}{\gamma _{v}}\left( vt-x\right)
+\frac{2L}{1+v}+vt,t\right) ,\text{ \ \ \ \ } & x\in \left(
L+vt,L_{2}+vt\right) .%
\end{array}%
\right.  \label{phix+}
\end{equation}%
On the other hand, the time derivative is extended as follows%
\begin{equation}
\tilde{\phi}_{t}(x,t)=\left\{ 
\begin{array}{ll}
\phi _{t}\left( x,t\right) , & x\in \left( vt,L+vt\right) ,\smallskip \\ 
\frac{-1}{\gamma _{v}}\phi _{t}\left( \frac{1}{\gamma _{v}}\left(
vt-x\right) +\frac{2L}{1+v}+vt,t\right) ,\text{ \ \ \ \ } & x\in \left(
L+vt,L_{2}+vt\right) .%
\end{array}%
\right.  \label{phit+}
\end{equation}

Let us introduce the following family of Hilbert spaces%
\begin{equation*}
\mathcal{H}_{L+vt}\left( \mathbf{I}_{t}\right) :=\left\{ w\in H^{1}\left( 
\mathbf{I}_{t}\right) \text{, }w\left( L+vt\right) =0\right\} ,\ \ \ \text{\
for }t\geq 0
\end{equation*}%
and we assume that the initial data satisfies%
\begin{equation}
\phi ^{0}\in \mathcal{H}_{L}\left( \mathbf{I}_{0}\right) ,\text{ \ }\phi
^{1}\in L^{2}\left( \mathbf{I}_{0}\right) .  \label{ic}
\end{equation}

Now, we are ready to state the following existence and uniqueness result for
Problem (\ref{waveq}).

\begin{theorem}
\label{thexist1}Under the assumptions \emph{(\ref{tlike}) }and \emph{(\ref%
{ic}), }the solution of\emph{\ }Problem \emph{(\ref{waveq})}%
\begin{equation}
\phi \in C\left( [0,T];\mathcal{H}_{L+vt}\left( \mathbf{I}_{t}\right)
\right) \cap C^{1}\left( [0,T];L^{2}\left( \mathbf{I}_{t}\right) \right) ,
\label{regul}
\end{equation}%
\emph{\ }is given by the series \emph{(\ref{exact0})} where the coefficients 
$a_{n}\in 
%TCIMACRO{\U{2102} }%
%BeginExpansion
\mathbb{C}
%EndExpansion
$ are given by the following formula 
\begin{equation}
a_{n}=\frac{1}{4\gamma _{\eta }\omega _{n}}\int\limits_{0}^{L_{2}}\left( 
\tilde{\phi}_{x}^{0}+\tilde{\phi}^{1}\right) e^{-\frac{1-v}{L}\omega _{n}x}dx%
\text{, \ \ \ for\ }n\in 
%TCIMACRO{\U{2124} }%
%BeginExpansion
\mathbb{Z}
%EndExpansion
\   \label{an+}
\end{equation}%
and we have%
\begin{equation}
\sum_{n\in \mathbb{%
%TCIMACRO{\U{2124} }%
%BeginExpansion
\mathbb{Z}
%EndExpansion
}}\left\vert \omega _{n}a_{n}\right\vert ^{2}=\frac{L}{8\left( 1-v\right)
\gamma _{\eta }^{2}}\int\limits_{vt}^{L_{2}+vt}e^{\frac{1-v}{L}\ln
\left\vert \gamma _{\eta }\right\vert (t+x)}\left( \tilde{\phi}_{x}+\tilde{%
\phi}_{t}\right) ^{2}dx<+\infty ,  \label{nan+}
\end{equation}%
where\emph{\ }$\tilde{\phi}_{x}^{0}$ and $\tilde{\phi}^{1}$ are extensions
of the initial data $\phi ^{0}$and\ $\phi ^{1}$ on the interval $\left(
0,L_{2}\right) $ given above by \emph{(\ref{phix+}) }and \emph{(\ref{phit+}) 
}respectively.
\end{theorem}

\begin{proof}
The general solution of (\ref{waveq}) is given by D'Alembert's formula 
\begin{equation*}
\phi (x,t)=f(t+x)+g\left( t-x\right) ,
\end{equation*}
where $f$ and $g$ are arbitrary continuous functions. Let us check the
boundary conditions. On one hand, at the left endpoint we have%
\begin{equation*}
\left( 1-\eta v\right) f^{\prime }(\left( 1+v\right) t)-g^{\prime }\left(
\left( 1-v\right) t\right) =\left( \eta -v\right) \left( f^{\prime }(\left(
1+v\right) t)+g^{\prime }\left( \left( 1-v\right) t\right) \right) .
\end{equation*}%
Setting $z=\left( 1-v\right) t$, we obtain%
\begin{equation}
\left( 1-\eta v-\eta +v\right) f^{\prime }\left( \gamma _{v}z\right) =\left(
1-\eta v+\eta -v\right) g^{\prime }\left( z\right) .  \label{+a1}
\end{equation}%
On the other hand, at the right endpoint, we infer that%
\begin{equation*}
f(\left( 1+v\right) t+L)+g\left( \left( 1-v\right) t-L\right) =0.
\end{equation*}%
Denoting $y=z-L$, we obtain 
\begin{equation}
f(\gamma _{v}y+\frac{2L}{1-v})=-g\left( y\right) .  \label{+b1}
\end{equation}%
Then, taking the derivative with respect to $y,$ we have 
\begin{equation}
\gamma _{v}f^{\prime }(\gamma _{v}y+\frac{2L}{1-v})=-g^{\prime }\left(
y\right) .  \label{+Db1}
\end{equation}%
Thanks to (\ref{+a1}), (\ref{+Db1}) and taking $\xi =\gamma _{v}y,$ we can
deduce that $f^{\prime }$\ satisfies 
\begin{equation}
\left( 1-\eta v+\eta -v\right) \gamma _{v}f^{\prime }\left( \xi +\frac{2L}{%
1-v}\right) =-\left( 1-\eta v-\eta +v\right) f^{\prime }\left( \xi \right) .
\end{equation}%
After few simplifications, this can be written as 
\begin{equation}
f^{\prime }\left( \xi +\frac{2L}{1-v}\right) =-\frac{f^{\prime }\left( \xi
\right) }{\gamma _{\eta }}.  \label{a+b}
\end{equation}

This formula suggests that $f^{\prime }(\xi )=e^{\beta \xi }$, for some
parameter $\beta $ to be determined later. Substituting the form $e^{\beta
\xi }$\ in (\ref{a+b}), we infer that $e^{\beta \left( \xi +\frac{2L}{1-v}%
\right) }=-e^{\beta \xi }/\gamma _{\eta }.$ Then we have:

$\bullet $ If $0\leq \eta <1$, then $\gamma _{\eta }\geq 1$ and we get 
\begin{equation*}
e^{\beta \left( \xi +\frac{2L}{1-v}\right) }=-e^{-\ln \gamma _{\eta
}}e^{\beta \xi }.
\end{equation*}%
Solving this equation for $\beta $, we obtain a sequence of values $\beta
_{n}=\left( 1-v\right) \omega _{n}/L,n\in 
%TCIMACRO{\U{2124} }%
%BeginExpansion
\mathbb{Z}
%EndExpansion
,$ where 
\begin{equation*}
\omega _{n}=\frac{1}{2}\left( 2n+1\right) i\pi -\frac{1}{2}\ln \gamma _{\eta
}.
\end{equation*}

$\bullet $ If $\eta >1$, we have $\gamma _{\eta }<-1$ and we obtain another
sequence of values $\beta _{n}=\left( 1-v\right) \omega _{n}/L,n\in 
%TCIMACRO{\U{2124} }%
%BeginExpansion
\mathbb{Z}
%EndExpansion
,$ where this time 
\begin{equation*}
\omega _{n}=ni\pi -\frac{1}{2}\ln \left\vert \gamma _{\eta }\right\vert .
\end{equation*}%
Note that in both cases we have $\ln \left\vert \gamma _{\eta }\right\vert
\geq 1$ and the real part of $\omega _{n}$ is negative.

Due to the superposition principal, it follows that $f$ can be written as%
\begin{equation*}
f^{\prime }\left( \xi \right) =\sum_{n\in \mathbb{Z}}c_{n}e^{\frac{1-v}{L}%
\omega _{n}\xi },\text{ \ \ \ \ }c_{n}\in 
%TCIMACRO{\U{2102} }%
%BeginExpansion
\mathbb{C}
%EndExpansion
,
\end{equation*}%
where $c_{n}$ are\ complex coefficients to be determined later. The function 
$f$ can be written as%
\begin{equation*}
f\left( \xi \right) =c+\sum_{n\in \mathbb{Z}}\frac{Lc_{n}}{\left( 1-v\right)
\omega _{n}}e^{\frac{1-v}{L}\omega _{n}\xi },
\end{equation*}%
for some constant $c.$ Using (\ref{+b1}), we deduce that 
\begin{equation}
g\left( \xi \right) =-c+\sum_{n\in \mathbb{Z}}\frac{Lc_{n}}{\left(
1-v\right) \gamma _{\eta }\omega _{n}}e^{\frac{1+v}{L}\omega _{n}\xi },
\label{g}
\end{equation}%
where we have used the fact that $e^{2\omega _{n}}=-1/\gamma _{\eta }$
whether $0\leq \eta <1\ $or $\eta >1.$

Thanks to D'Alembert's formula, the solution of Problem (\ref{waveq}) is
given by the series 
\begin{equation}
\phi (x,t)=\sum_{n\in \mathbb{Z}}\frac{L}{\left( 1-v\right)\gamma _{\eta }
\omega _{n}}c_{n}\left( \gamma _{\eta }e^{\frac{1-v}{L}\omega _{n}(t+x)}+e^{%
\frac{1+v}{L}\omega _{n}(t-x)}\right) .  \label{phi_cn}
\end{equation}%
To obtain (\ref{exact0}), we set%
\begin{equation}
a_{n}:=\frac{L}{\left( 1-v\right)\gamma _{\eta } \omega _{n}}c_{n}.
\label{an}
\end{equation}%
The coefficient $c_{n}$ are determined as follows. Going back to (\ref%
{phi_cn}), we infer that%
\begin{eqnarray}
\phi _{x}(x,t) &=&\sum_{n\in 
%TCIMACRO{\U{2124} }%
%BeginExpansion
\mathbb{Z}
%EndExpansion
}c_{n}\left( e^{\frac{1-v}{L}\omega _{n}(t+x)}-\frac{\gamma _{v}}{\gamma
_{\eta }}e^{\frac{1+v}{L}\omega _{n}(t-x)}\right) ,\smallskip   \label{ph x}
\\
\phi _{t}(x,t) &=&\sum_{n\in 
%TCIMACRO{\U{2124} }%
%BeginExpansion
\mathbb{Z}
%EndExpansion
}c_{n}\left( e^{\frac{1-v}{L}\omega _{n}(t+x)}+\frac{\gamma _{v}}{\gamma
_{\eta }}e^{\frac{1+v}{L}\omega _{n}\text{ }(t-x)}\right) ,  \label{ph t}
\end{eqnarray}%
for $x\in \left( vt,L+vt\right) $ and $t\geq 0$. It follows from (\ref{phix+}%
) and (\ref{phit+}) that the extensions $\tilde{\phi}_{x}$ and $\tilde{\phi}%
_{t}$ are given by%
\begin{eqnarray}
\tilde{\phi}_{x}(x,t)=\left\{ 
\begin{array}{l}
\displaystyle\sum_{n\in 
%TCIMACRO{\U{2124} }%
%BeginExpansion
\mathbb{Z}
%EndExpansion
}c_{n}\left( e^{\frac{1-v}{L}\omega _{n}(t+x)}-\frac{\gamma _{v}}{\gamma
_{\eta }}e^{\frac{1+v}{L}\omega _{n}(t-x)}\right) ,\text{\ \ if }x\in \left(
vt,L+vt\right) ,\smallskip  \\ 
\frac{1}{\gamma _{v}}\displaystyle\sum_{n\in 
%TCIMACRO{\U{2124} }%
%BeginExpansion
\mathbb{Z}
%EndExpansion
}c_{n}\left( e^{\frac{1-v}{L}\omega _{n}(\left( 1+v\right) t+\frac{vt-x}{%
\gamma _{v}}+\frac{2L}{1+v})}\right.  \\ 
\text{ \ \ \ \ \ \ \ \ }-\displaystyle\frac{\gamma _{v}}{\gamma _{\eta
}}\left. e^{\frac{1+v}{L}\omega _{n}(\left( 1-v\right) t-\frac{vt-x}{\gamma
_{v}}-\frac{2L}{1+v})}\right) ,\text{\ \ if }x\in \left(
L+vt,L_{2}+vt\right) ,%
\end{array}%
\right.   \label{phx}\smallskip \\
\tilde{\phi}_{t}(x,t)=\left\{ 
\begin{array}{l}
\displaystyle\sum_{n\in 
%TCIMACRO{\U{2124} }%
%BeginExpansion
\mathbb{Z}
%EndExpansion
}c_{n}\left( e^{\frac{1-v}{L}\omega _{n}(t+x)}+\frac{\gamma _{v}}{\gamma
_{\eta }}e^{\frac{1+v}{L}\omega _{n}(t-x)}\right) ,\text{ \ if }x\in \left(
vt,L+vt\right) ,\smallskip  \\ 
\frac{-1}{\gamma _{v}}\displaystyle\sum_{n\in 
%TCIMACRO{\U{2124} }%
%BeginExpansion
\mathbb{Z}
%EndExpansion
}c_{n}\left( e^{\frac{1-v}{L}\omega _{n}(\left( 1+v\right) t+\frac{vt-x}{%
\gamma _{v}}+\frac{2L}{1+v})}\right.  \\ 
\text{ \ \ \ \ \ \ \ \ }\displaystyle+\frac{\gamma _{v}}{\gamma _{\eta }}%
\displaystyle\left. e^{\frac{1+v}{L}\omega _{n}(\left( 1-v\right) t-\frac{%
vt-x}{\gamma _{v}}-\frac{2L}{1+v})}\right) ,\text{\ \ if }x\in \left(
L+vt,L_{2}+vt\right) .%
\end{array}%
\right.   \label{pht}
\end{eqnarray}%
Taking the sum of (\ref{phx}) and (\ref{pht}), we get%
\begin{equation*}
\tilde{\phi}_{x}+\tilde{\phi}_{t}=\left\{ 
\begin{array}{ll}
2\displaystyle\sum_{n\in 
%TCIMACRO{\U{2124} }%
%BeginExpansion
\mathbb{Z}
%EndExpansion
}c_{n}e^{\frac{1-v}{L}\omega _{n}(t+x)}\smallskip , & x\in \left(
vt,L+vt\right) ,\smallskip  \\ 
\frac{-2}{\gamma _{v}}\displaystyle\sum_{n\in 
%TCIMACRO{\U{2124} }%
%BeginExpansion
\mathbb{Z}
%EndExpansion
}c_{n}\frac{\gamma _{v}}{\gamma _{\eta }}e^{-2\omega _{n}}e^{\frac{1-v}{L}%
\omega _{n}(t+x)}, & x\in \left( L+vt,L_{2}+vt\right) .%
\end{array}%
\right. 
\end{equation*}%
Since $e^{-2\omega _{n}}=-\gamma _{\eta }$, then we have the unified
expression%
\begin{equation}
\tilde{\phi}_{x}+\tilde{\phi}_{t}=2\sum_{n\in 
%TCIMACRO{\U{2124} }%
%BeginExpansion
\mathbb{Z}
%EndExpansion
}c_{n}e^{\frac{1-v}{L}\omega _{n}(t+x)},\ \ \ \text{for }x\in \left(
vt,L_{2}+vt\right) \text{ and }t\geq 0.  \label{29}
\end{equation}%
Using the definition of $\omega _{n},$ we get%
\begin{equation}
\tilde{\phi}_{x}+\tilde{\phi}_{t}=\left\{ 
\begin{array}{ll}
\displaystyle2e^{\frac{1-v}{2L}\left( i\pi -\ln \gamma _{\eta }\right)
(t+x)}\sum_{n\in 
%TCIMACRO{\U{2124} }%
%BeginExpansion
\mathbb{Z}
%EndExpansion
}c_{n}e^{\frac{1-v}{L}ni\pi (t+x)}\smallskip , & \text{if }0\leq \eta <1, \\ 
\displaystyle2e^{-\frac{1-v}{2L}\ln \left\vert \gamma _{\eta }\right\vert
(t+x)}\sum_{n\in 
%TCIMACRO{\U{2124} }%
%BeginExpansion
\mathbb{Z}
%EndExpansion
}c_{n}e^{\frac{1-v}{L}ni\pi (t+x)}, & \text{if }\eta >1.%
\end{array}%
\right.   \label{29+}
\end{equation}%
Taking into account that $\left\{ e^{\frac{n\pi
i\left( 1-v\right) }{L}\left( t+x\right) }/\sqrt{L_{2}}\right\} _{n\in \mathbb{%
%TCIMACRO{\U{2124} }%
%BeginExpansion
\mathbb{Z}
%EndExpansion
}}$ is an orthonormal basis for $L^{2}\left( vt,L_{2}+vt\right) $, for every 
$t\geq 0$, we rewrite (\ref{29}) as 
\begin{equation}
\sum_{n\in 
%TCIMACRO{\U{2124} }%
%BeginExpansion
\mathbb{Z}
%EndExpansion
}c_{n}\frac{e^{\frac{1-v}{L}ni\pi (t+x)}}{\sqrt{L_{2}}}=\left\{ 
\begin{array}{ll}
\displaystyle\frac{1}{2\sqrt{L_{2}}}e^{-\frac{1-v}{2L}\left( i\pi -\ln
\gamma _{\eta }\right) (t+x)}\left( \tilde{\phi}_{x}+\tilde{\phi}_{t}\right)
\smallskip , & \text{if }0\leq \eta <1, \\ 
\displaystyle\frac{1}{2\sqrt{L_{2}}}e^{\frac{1-v}{2L}\ln \left\vert \gamma
_{\eta }\right\vert (t+x)}\left( \tilde{\phi}_{x}+\tilde{\phi}_{t}\right) ,
& \text{if }\eta >1.%
\end{array}%
\right.   \label{30}
\end{equation}%
By consequence,%
\begin{equation*}
c_{n}=\left\{ 
\begin{array}{ll}
\displaystyle\frac{1}{2L_{2}}\int\limits_{vt}^{L_{2}+vt}e^{-\frac{1-v}{2L}%
\left( i\pi -\ln \gamma _{\eta }\right) (t+x)}\left( \tilde{\phi}_{x}+\tilde{%
\phi}_{t}\right) e^{-\frac{n\pi i\left( 1-v\right) }{L}\left( t+x\right)
}dx\smallskip , & \text{if }0\leq \eta <1, \\ 
\displaystyle\frac{1}{2L_{2}}\int\limits_{vt}^{L_{2}+vt}e^{\frac{1-v}{2L}\ln
\left\vert \gamma _{\eta }\right\vert (t+x)}\left( \tilde{\phi}_{x}+\tilde{%
\phi}_{t}\right) e^{-\frac{n\pi i\left( 1-v\right) }{L}\left( t+x\right) }dx,
& \text{if }\eta >1,%
\end{array}%
\right. 
\end{equation*}%
for\ $n\in 
%TCIMACRO{\U{2124} }%
%BeginExpansion
\mathbb{Z}
%EndExpansion
$. Whether $0\leq \eta <1$ or $\eta >1,$ in both cases, we have%
\begin{equation*}
c_{n}=\frac{1}{2L_{2}}\int\limits_{vt}^{L_{2}+vt}\left( \tilde{\phi}_{x}+%
\tilde{\phi}_{t}\right) e^{-\frac{\left( 1-v\right) }{L}\omega _{n}\left(
t+x\right) }dx\text{, \ \ \ for\ }n\in 
%TCIMACRO{\U{2124} }%
%BeginExpansion
\mathbb{Z}
%EndExpansion
\text{.}
\end{equation*}%
For $t=0$ and tacking (\ref{an})\ into account, we obtain (\ref{an+}) as
claimed.

Moreover, as a consequence of Parseval's equality, it comes that%
\begin{equation*}
\sum_{n\in 
%TCIMACRO{\U{2124} }%
%BeginExpansion
\mathbb{Z}
%EndExpansion
}\left\vert c_{n}\right\vert ^{2}=\left\{ 
\begin{array}{ll}
\displaystyle\frac{1}{4L_{2}}\int\limits_{vt}^{L_{2}+vt}\left\vert e^{-\frac{%
1-v}{2L}\left( i\pi -\ln \gamma _{\eta }\right) (t+x)}\right\vert ^{2}\left( 
\tilde{\phi}_{x}+\tilde{\phi}_{t}\right) ^{2}dx\smallskip , & \text{if }%
0\leq \eta <1, \\ 
\displaystyle\frac{1}{4L_{2}}\int\limits_{vt}^{L_{2}+vt}\left\vert e^{\frac{%
1-v}{2L}\ln \left\vert \gamma _{\eta }\right\vert (t+x)}\right\vert
^{2}\left( \tilde{\phi}_{x}+\tilde{\phi}_{t}\right) ^{2}dx, & \text{if }\eta
>1.%
\end{array}%
\right.
\end{equation*}%
Whether $0\leq \eta <1$ or $\eta >1,$ it follows that%
\begin{equation*}
\sum_{n\in 
%TCIMACRO{\U{2124} }%
%BeginExpansion
\mathbb{Z}
%EndExpansion
}\left\vert \omega _{n}a_{n}\right\vert ^{2}=\frac{L^{2}}{\gamma _{\eta
}^{2}\left( 1-v\right) ^{2}}\sum_{n\in 
%TCIMACRO{\U{2124} }%
%BeginExpansion
\mathbb{Z}
%EndExpansion
}\left\vert c_{n}\right\vert ^{2}=\frac{L}{8\gamma _{\eta }^{2}\left(
1-v\right) }\int\limits_{vt}^{L_{2}+vt}e^{\frac{1-v}{L}\ln \left\vert \gamma
_{\eta }\right\vert (t+x)}\left( \tilde{\phi}_{x}+\tilde{\phi}_{t}\right)
^{2}dx.
\end{equation*}%
Thanks to (\ref{ic}), $\phi ^{0}$ and\ $\phi ^{1}$ belongs to $%
L^{2}\left( 0,L_{2}\right) $. Thus, the integral at the right hand side for $%
t=0$ is finite and 
\begin{equation*}
\sum_{n\in 
%TCIMACRO{\U{2124} }%
%BeginExpansion
\mathbb{Z}
%EndExpansion
}\left\vert \omega _{n}a_{n}\right\vert ^{2}<+\infty .
\end{equation*}%
Recalling that $\left\vert w_{n}\right\vert ^{2}=O\left( n^{2}\right) $, for
large values of $n,$ then%
\begin{equation}
\sum_{n\in \mathbb{Z}}\left\vert na_{n}\right\vert ^{2}<+\infty .
\label{nan}
\end{equation}

Let $T>0$ and $t\in \left[ 0,T\right] .$ Due to the continuity of the
exponential function, we get%
\begin{equation*}
\left\vert a_{n}\left( \gamma _{\eta }e^{\frac{1-v}{L}\omega _{n}(t+x)}+e^{%
\frac{1+v}{L}\omega _{n}(t-x)}\right) \right\vert \leq C_{T}\left\vert
a_{n}\right\vert ,
\end{equation*}%
where $C_{T}$ is a constant depending only on $v,\eta ,L$ and $T.$

Going back to (\ref{ph x}), (\ref{ph t}) and due to (\ref{an}), we can check
that 
\begin{equation*}
\left\vert c_{n}\left( e^{\frac{1-v}{L}\omega _{n}(t+x)}\pm \frac{\gamma _{v}%
}{\gamma _{\eta }}e^{\frac{1+v}{L}\omega _{n}(t-x)}\right) \right\vert \leq
C_{T}^{\prime }\left\vert na_{n}\right\vert ,
\end{equation*}%
for some constant $C_{T}^{\prime }$.

Taking (\ref{nan}) into account, we infer that $\phi (x,t)$, $\phi _{x}(x,t)$
and $\phi _{t}(x,t)$ belong to $L^{2}\left( \mathbf{I}_{t}\right) $, for $%
t\geq 0$. In particular, $\phi (x,t)\in \mathcal{H}_{L+vt}\left( \mathbf{I}%
_{t}\right) $, for $t\geq 0$. The continuity in time of $\phi $ and $\phi
_{t}$ as functions of $t$ with values in $\mathcal{H}_{L+vt}\left( \mathbf{I}%
_{t}\right) $ and $L^{2}\left( \mathbf{I}_{t}\right) $, respectively,
follows as they are the sums of uniformly converging series of continuous
functions. This shows (\ref{regul}).
\end{proof}

\section{A conserved quantity for the string with no damper}

For the undamped case, i.e. $\eta =0$ in Problem (\ref{waveq}), we show that
the energy $\mathcal{E}_{v}$ given by (\ref{QE}) is conserved in time.

\begin{theorem}
\label{thcv}Under the assumptions \emph{(\ref{tlike})} and \emph{(\ref{ic}),}
the solution of Problem \emph{(\ref{waveq})} satisfies%
\begin{equation}
\mathcal{E}_{v}\left( t\right) =\frac{\pi ^{2}\left( 1-v^{2}\right) }{2L}%
\sum_{n\in \mathbb{%
%TCIMACRO{\U{2124} }%
%BeginExpansion
\mathbb{Z}
%EndExpansion
}}\left\vert \left( 2n+1\right) a_{n}\right\vert ^{2},\ \ \ \ \ \text{for\ }%
t\geq 0,  \label{con-qua}
\end{equation}%
where the left hand side is independent of $t.$
\end{theorem}

\begin{proof}
If $\eta =0,$ then $\omega _{n}=\left( 2n+1\right) i\pi /2$ and the identity
(\ref{nan+}) becomes%
\begin{equation}
\frac{1}{1-v}\int\limits_{vt}^{L_{2}+vt}\left( \tilde{\phi}_{x}+\tilde{\phi}%
_{t}\right) ^{2}dx=\frac{2\pi ^{2}}{L}\sum_{n\in \mathbb{%
%TCIMACRO{\U{2124} }%
%BeginExpansion
\mathbb{Z}
%EndExpansion
}}\left\vert \left( 2n+1\right) a_{n}\right\vert ^{2}.  \label{eq-per}
\end{equation}%
Using the extensions (\ref{phix+}), (\ref{phit+}) and considering the change
of variable 
\begin{equation*}
x=\gamma _{v}\left( vt-\xi \right) +\frac{2L}{1-v}+vt,
\end{equation*}
in $\left( L+vt,L_{2}+vt\right) ,$ then we have 
\begin{equation*}
\frac{1}{1-v}\int\limits_{L+vt}^{L_{2}+vt}\left( \tilde{\phi}_{x}\left(
x,t\right) +\tilde{\phi}_{t}\left( x,t\right) \right) ^{2}dx=\frac{1}{1+v}%
\int_{\mathbf{I}_{t}}\left( \phi _{x}\left( \xi ,t\right) -\phi _{t}\left(
\xi ,t\right) \right) ^{2}d\xi .
\end{equation*}%
Taking (\ref{eq-per}) into account,\ it comes that 
\begin{eqnarray*}
\frac{1}{1-v}\int\limits_{vt}^{L_{2}+vt}\left( \tilde{\phi}_{x}+\tilde{\phi}%
_{t}\right) ^{2}dx &=&\frac{1}{1-v}\int_{\mathbf{I}_{t}}\left( \phi
_{t}+\phi _{x}\right) ^{2}dx+\frac{1}{1+v}\int_{\mathbf{I}_{t}}\left( \phi
_{x}-\phi _{t}\right) ^{2}dx \\
&=&\frac{2\pi ^{2}}{L}\sum_{n\in \mathbb{%
%TCIMACRO{\U{2124} }%
%BeginExpansion
\mathbb{Z}
%EndExpansion
}}\left\vert \left( 2n+1\right) a_{n}\right\vert ^{2}.
\end{eqnarray*}%
Expanding $\left( \phi _{x}\pm \phi _{t}\right) ^{2}$ and collecting similar
terms, we get 
\begin{equation}
\frac{1}{1-v^{2}}\left( \int_{\mathbf{I}_{t}}\phi _{x}^{2}+\phi
_{t}^{2}+2v\phi _{x}\phi _{t}dx\right) =\frac{\pi ^{2}}{L}\sum_{n\in \mathbb{%
%TCIMACRO{\U{2124} }%
%BeginExpansion
\mathbb{Z}
%EndExpansion
}}\left\vert \left( 2n+1\right) a_{n}\right\vert ^{2},\text{ \ \ for}\ t\geq
0.  \label{est}
\end{equation}%
The left hand side is equal to $2\mathcal{E}_{v}\left( t\right) /\left(
1-v^{2}\right) $ and (\ref{con-qua}) follows.
\end{proof}

\begin{remark}
Using Leibnitz's rule for differentiation
under the integral sign, we can check directly that $\frac{d}{dt}\mathcal{E}%
_{v}\left( t\right) =0$, see the appendix.
\end{remark}
\begin{remark}
The energy expression $\mathcal{E}%
_{v}\left( t\right)$ is also shown to be conserved in time for the Dirichlet boundary conditions at both ends, see \cite{GhSe2022}.
\end{remark}
Let us now compare $\mathcal{E}_{v}\left( t\right) $ to the usual expression
of energy $E_{v}\left( t\right) $ for the wave equation

\begin{corollary}
\label{coro3}Under the assumptions \emph{(\ref{tlike}) }and \emph{(\ref{ic}),%
} the energy $E_{v}\left( t\right) $ of the solution of undamped Problem 
\emph{(\ref{waveq})} satisfies%
\begin{equation}
\frac{\mathcal{E}_{v}\left( 0\right) }{1+v}\leq E_{v}\left( t\right) \leq 
\frac{\mathcal{E}_{v}\left( 0\right) }{1-v},\ \ \ \ \ for\ t\geq 0
\label{ES0}
\end{equation}%
and%
\begin{equation}
\frac{1}{\gamma _{v}}E_{v}\left( 0\right) \leq E_{v}\left( t\right) \leq
\gamma _{v}E_{v}\left( 0\right) ,\ \ \ \ \ \text{\ for }t\geq 0.
\label{stab}
\end{equation}
\end{corollary}

\begin{proof}
It suffices to argue as in the proof of Corollary 2 in \cite{GhSe2022}.
\end{proof}

\begin{remark}
The solution $\phi $ given by (\ref{exact0}), with\emph{\ }$\eta =0,$\emph{\ 
}satisfies the periodicity relation%
\begin{equation}
\phi (x+vT_{v},t+T_{v})=-\phi (x,t)\text{, \ \ \ }t\geq 0.  \label{T-period}
\end{equation}%
It follows in particular that the energy $E_{v}$ is a $T_{v}-$periodic
function in time.
\end{remark}

\begin{remark}
\label{rmksharp}The equality in (\ref{ES0}) holds at least if $\phi
_{t}\left( x,t_{0}\right) =\pm \phi _{x}\left( x,t_{0}\right) $, for $x\in 
\mathbf{I}_{t_{0}}$ and some $t_{0}\geq 0$. Indeed, we have%
\begin{equation*}
\mathcal{E}_{v}\left( t_{0}\right) =E_{v}\left( t_{0}\right) \pm
v\int_{vt_{0}}^{L+vt_{0}}\phi _{x}\left( x,t_{0}\right) \phi _{t}\left(
x,t_{0}\right) \ dx=\left( 1\pm v\right) E_{v}\left( t_{0}\right) ,
\end{equation*}%
i.e. $E_{v}\left( t_{0}\right) =\mathcal{E}_{v}\left( t_{0}\right) /\left(
1\pm v\right) $. The + and -- signs are used respectively.
\end{remark}

\begin{remark}
Let $0<v<1$. Arguing as in \cite{GhSe2022,Seng2020}, we can show the identity%
\begin{equation}
\frac{1}{v^{2}}\int\limits_{0}^{T_{v}}\phi _{x}^{2}(vt,t)dt=\frac{\mathcal{E}%
_{v}\left( 0\right) }{\left( 1-v^{2}\right) ^{2}}=\int\limits_{0}^{T_{v}}%
\phi _{x}^{2}(L+vt,t)dt.
\end{equation}%
Then, we can derive exact boundary observability results at each endpoint for $T\geq
T_{v}=2L/\left( 1-v^{2}\right) $.
\end{remark}

\section{Exponential stability for the string with a boundary damper}

In this section, we keep $0\leq v<1$ and assume that $\eta >0$.

\begin{theorem}
\label{th-stabl}Under the assumptions \emph{(\ref{tlike})} and \emph{(\ref%
{ic}),} the solution of Problem \emph{(\ref{waveq})} satisfies%
\begin{multline}
\frac{1}{1+v}\int_{\mathbf{I}_{t}}e^{\frac{1+v}{L}\ln \left\vert \gamma
_{\eta }\right\vert (t-x)}\left( \phi _{x}-\phi _{t}\right) ^{2}dx
\label{Sl} \\
+\frac{1}{\gamma _{\eta }^{2}\left( 1-v\right) }\int_{\mathbf{I}_{t}}e^{%
\frac{1-v}{L}\ln \left\vert \gamma _{\eta }\right\vert (t+x)}\left( \phi
_{t}+\phi _{x}\right) ^{2}dx=\frac{8}{L}\sum_{n\in \mathbb{%
%TCIMACRO{\U{2124} }%
%BeginExpansion
\mathbb{Z}
%EndExpansion
}}\left\vert \omega _{n}a_{n}\right\vert ^{2},
\end{multline}%
where the left hand side is finite and independent of $t$. Moreover, it
holds that%
\begin{equation}
M_{1}e^{-\frac{1-v^{2}}{L}\ln \left\vert \gamma _{\eta }\right\vert t}\leq
E_{v}\left( t\right) \leq M_{2}e^{-\frac{1-v^{2}}{L}\ln \left\vert \gamma
_{\eta }\right\vert t},\text{ for}\ t\geq 0,  \label{est0}
\end{equation}%
where 
\begin{eqnarray*}
M_{1} &:&=\frac{2}{L}\min \left( 1+v,\left\vert \gamma _{\eta }\right\vert
^{\left( 1+v\right) }\left( 1-v\right) \right) \sum_{n\in \mathbb{%
%TCIMACRO{\U{2124} }%
%BeginExpansion
\mathbb{Z}
%EndExpansion
}}\left\vert \omega _{n}a_{n}\right\vert ^{2}, \\
M_{2} &:&=\frac{2}{L}\max \left( \left\vert \gamma _{\eta }\right\vert
^{\left( 1+v\right) }\left( 1+v\right) ,\gamma _{\eta }^{2}\left( 1-v\right)
\right) \sum_{n\in \mathbb{%
%TCIMACRO{\U{2124} }%
%BeginExpansion
\mathbb{Z}
%EndExpansion
}}\left\vert \omega _{n}a_{n}\right\vert ^{2}.
\end{eqnarray*}
\end{theorem}

\begin{proof}
Let us split the integral in the identity (\ref{nan+}) to the integrals%
\begin{equation}
\int\limits_{vt}^{L_{2}+vt}=\int_{\mathbf{I}_{t}}+\int%
\limits_{L+vt}^{L_{2}+vt},  \label{32a}
\end{equation}%
then considering the change of variable $x=\gamma _{v}\left( vt-\xi \right) +%
\frac{2L}{1-v}+vt$ in $\left( L+vt,L_{2}+vt\right) $, we obtain%
\begin{multline}
\frac{1}{\left( 1-v\right) \gamma _{\eta }^{2}}\int%
\limits_{L+vt}^{L_{2}+vt}e^{\frac{1-v}{L}\ln \left\vert \gamma _{\eta
}\right\vert (t+x)}\left( \tilde{\phi}_{x}\left( x,t\right) +\tilde{\phi}%
_{t}\left( x,t\right) \right) ^{2}dx  \label{32b} \\
=-\frac{1}{\left( 1-v\right) \gamma _{\eta }^{2}}\int%
\limits_{L+vt}^{L_{2}+vt}\frac{\gamma _{v}}{\gamma _{v}^{2}}e^{\frac{1-v}{L}%
\ln \left\vert \gamma _{\eta }\right\vert (t+\gamma _{v}\left( vt-\xi
\right) +vt)}e^{2\ln \left\vert \gamma _{\eta }\right\vert }\left( \tilde{%
\phi}_{x}\left( \xi ,t\right) -\tilde{\phi}_{t}\left( \xi ,t\right) \right)
^{2}d\xi \\
=\frac{1}{1+v}\int_{\mathbf{I}_{t}}e^{\frac{1+v}{L}\ln \left\vert \gamma
_{\eta }\right\vert (t-\xi )}\left( \phi _{x}\left( \xi ,t\right) -\phi
_{t}\left( \xi ,t\right) \right) ^{2}d\xi .
\end{multline}
We used the definition of the extensions (\ref{phix+}), (\ref{phit+}) and
the fact that $e^{2\ln \left\vert \gamma _{\eta }\right\vert }=\gamma _{\eta
}^{2}.$ Then, combining (\ref{32a}) and (\ref{32b}) we obtain (\ref{Sl}).

Expanding $\left( \phi _{x}\pm \phi _{t}\right) ^{2}$ and collecting similar
terms, we get%
\begin{multline}
\int_{\mathbf{I}_{t}}\left( \frac{1}{1+v}e^{\frac{1+v}{L}\ln \left\vert
\gamma _{\eta }\right\vert (t-x)}+\frac{1}{\gamma _{\eta }^{2}\left(
1-v\right) }e^{\frac{1-v}{L}\ln \left\vert \gamma _{\eta }\right\vert
(t+x)}\right) \left( \phi _{t}^{2}+\phi _{x}^{2}\right) dx  \label{=E0} \\
-2\int_{\mathbf{I}_{t}}\left( \frac{1}{1+v}e^{\frac{1+v}{L}\ln \left\vert
\gamma _{\eta }\right\vert (t-x)}-\frac{1}{\gamma _{\eta }^{2}\left(
1-v\right) }e^{\frac{1-v}{L}\ln \left\vert \gamma _{\eta }\right\vert
(t+x)}\right) \phi _{t}\phi _{x}dx \\
=\frac{8}{L}\sum_{n\in \mathbb{%
%TCIMACRO{\U{2124} }%
%BeginExpansion
\mathbb{Z}
%EndExpansion
}}\left\vert \omega _{n}a_{n}\right\vert ^{2}.
\end{multline}%
For $vt\leq x\leq L+vt$ and $t\geq 0,$ let us denote%
\begin{equation*}
A\left( x,t\right) =\frac{1}{1+v}e^{\frac{1+v}{L}\ln \left\vert \gamma
_{\eta }\right\vert (t-x)}\text{ \ and \ }B\left( x,t\right) =\frac{1}{%
\gamma _{\eta }^{2}\left( 1-v\right) }e^{\frac{1-v}{L}\ln \left\vert \gamma
_{\eta }\right\vert (t+x)}.
\end{equation*}%
Then, we can rewrite (\ref{=E0}) as 
\begin{equation*}
\int_{\mathbf{I}_{t}}\left( A+B\right) \left( \phi _{t}^{2}+\phi
_{x}^{2}\right) dx-2\int_{\mathbf{I}_{t}}\left( A-B\right) \phi _{t}\phi
_{x}dx=\frac{8}{L}\sum_{n\in \mathbb{%
%TCIMACRO{\U{2124} }%
%BeginExpansion
\mathbb{Z}
%EndExpansion
}}\left\vert \omega _{n}a_{n}\right\vert ^{2}.
\end{equation*}%
Using the algebraic inequality%
\begin{equation*}
-\left\vert A-B\right\vert \left( \phi _{t}^{2}+\phi _{x}^{2}\right) \leq
\pm 2\left( A-B\right) \phi _{t}\phi _{x}\leq \left\vert A-B\right\vert
\left( \phi _{t}^{2}+\phi _{x}^{2}\right) ,
\end{equation*}%
we get%
\begin{multline*}
\int_{\mathbf{I}_{t}}\left( \left( A+B\right) -\left\vert A-B\right\vert
\right) \left( \phi _{t}^{2}+\phi _{x}^{2}\right) dx\leq \frac{2}{L}%
\sum_{n\in \mathbb{%
%TCIMACRO{\U{2124} }%
%BeginExpansion
\mathbb{Z}
%EndExpansion
}}\left\vert 2\omega _{n}a_{n}\right\vert ^{2} \\
\leq \int_{\mathbf{I}_{t}}\left( \left( A+B\right) +\left\vert
A-B\right\vert \right) \left( \phi _{t}^{2}+\phi _{x}^{2}\right) dx.
\end{multline*}%
Knowing that $\left( a+b\right) -\left\vert a-b\right\vert =2\min \left\{ a,b\right\}$ 
 and $\left( a+b\right) +\left\vert a-b\right\vert =2\max \left\{
a,b\right\}$, for $a,b\in 
%TCIMACRO{\U{211d} }%
%BeginExpansion
\mathbb{R}
%EndExpansion
,$ then the precedent estimation reads%
\begin{equation*}
\int_{\mathbf{I}_{t}}\min \left\{ A,B\right\} \left( \phi _{t}^{2}+\phi
_{x}^{2}\right) dx\leq \frac{4}{L}\sum_{n\in \mathbb{%
%TCIMACRO{\U{2124} }%
%BeginExpansion
\mathbb{Z}
%EndExpansion
}}\left\vert \omega _{n}a_{n}\right\vert ^{2}\leq \int_{\mathbf{I}_{t}}\max
\left\{ A,B\right\} \left( \phi _{t}^{2}+\phi _{x}^{2}\right) dx.
\end{equation*}%
Since $\ln \left\vert \gamma _{\eta }\right\vert \geq 0$ and $vt\leq x\leq
L+vt,$ we have 
\begin{equation*}
\frac{1}{1+v}e^{\frac{1+v}{L}\ln \left\vert \gamma _{\eta }\right\vert
(t-L-vt)}\leq A\left( x,t\right) \leq \frac{1}{1+v}e^{\frac{1+v}{L}\ln
\left\vert \gamma _{\eta }\right\vert (t-vt)},
\end{equation*}%
i.e.%
\begin{equation*}
\frac{e^{-\left( 1+v\right) \ln \left\vert \gamma _{\eta }\right\vert }}{1+v}%
e^{\frac{1-v^{2}}{L}\ln \left\vert \gamma _{\eta }\right\vert t}\leq A\left(
x,t\right) \leq \frac{1}{1+v}e^{\frac{1-v^{2}}{L}\ln \left\vert \gamma
_{\eta }\right\vert t}.
\end{equation*}%
Similarly, we obtain 
\begin{equation*}
\frac{1}{\gamma _{\eta }^{2}\left( 1-v\right) }e^{\frac{1-v^{2}}{L}\ln
\left\vert \gamma _{\eta }\right\vert t}\leq B\left( x,t\right) \leq \frac{%
e^{-\left( 1+v\right) \ln \left\vert \gamma _{\eta }\right\vert }}{1-v}e^{%
\frac{1-v^{2}}{L}\ln \left\vert \gamma _{\eta }\right\vert t}.
\end{equation*}%
It follows that%
\begin{multline*}
\min \left\{ \frac{1}{\left\vert \gamma _{\eta }\right\vert ^{\left(
1+v\right) }\left( 1+v\right) },\frac{1}{\gamma _{\eta }^{2}\left(
1-v\right) }\right\} E_{v}\left( t\right) \leq \frac{2e^{-\frac{1-v^{2}}{L}%
\ln \left\vert \gamma _{\eta }\right\vert t}}{L}\sum_{n\in \mathbb{%
%TCIMACRO{\U{2124} }%
%BeginExpansion
\mathbb{Z}
%EndExpansion
}}\left\vert \omega _{n}a_{n}\right\vert ^{2} \\
\leq \max \left\{ \frac{1}{1+v},\frac{1}{\left\vert \gamma _{\eta
}\right\vert ^{\left( 1+v\right) }\left( 1-v\right) }\right\} E_{v}\left(
t\right) ,
\end{multline*}%
hence%
\begin{multline*}
\left( \frac{2}{L}\sum_{n\in \mathbb{%
%TCIMACRO{\U{2124} }%
%BeginExpansion
\mathbb{Z}
%EndExpansion
}}\left\vert \omega _{n}a_{n}\right\vert ^{2}\right) \min \left\{
1+v,\left\vert \gamma _{\eta }\right\vert ^{\left( 1+v\right) }\left(
1-v\right) \right\} e^{-\frac{1-v^{2}}{L}\ln \left\vert \gamma _{\eta
}\right\vert t}\leq E_{v}\left( t\right)  \\
\leq \left( \frac{2}{L}\sum_{n\in \mathbb{%
%TCIMACRO{\U{2124} }%
%BeginExpansion
\mathbb{Z}
%EndExpansion
}}\left\vert \omega _{n}a_{n}\right\vert ^{2}\right) \max \left\{ \left\vert
\gamma _{\eta }\right\vert ^{\left( 1+v\right) }\left( 1+v\right) ,\gamma
_{\eta }^{2}\left( 1-v\right) \right\} e^{-\frac{1-v^{2}}{L}\ln \left\vert
\gamma _{\eta }\right\vert t}.
\end{multline*}%
This shows (\ref{est0}) and the theorem follows.
\end{proof}

\begin{remark}
If $v=0$ in (\ref{est0})$,$ then we get the decay estimate%
\begin{equation*}
\left( \frac{2}{L}\sum_{n\in \mathbb{%
%TCIMACRO{\U{2124} }%
%BeginExpansion
\mathbb{Z}
%EndExpansion
}}\left\vert \omega _{n}a_{n}\right\vert ^{2}\right) e^{-\frac{1}{L}\ln
\left\vert \gamma _{\eta }\right\vert t}\leq E_{0}\left( t\right) \leq
\gamma _{\eta }^{2}\left( \frac{2}{L}\sum_{n\in \mathbb{%
%TCIMACRO{\U{2124} }%
%BeginExpansion
\mathbb{Z}
%EndExpansion
}}\left\vert \omega _{n}a_{n}\right\vert ^{2}\right) e^{-\frac{1}{L}\ln
\left\vert \gamma _{\eta }\right\vert t},
\end{equation*}%
which is sharper then the estimate (\ref{estv=0}),
stated in the introduction.
\end{remark}

\begin{remark}
The constants in estimation (\ref{est0}) are (at least) asymptotically sharp
in the sense that if $\eta \rightarrow 0$, we recover the estimation (\ref%
{ES0}) with its sharp constants, see Remark \ref{rmksharp}.
\end{remark}

The next corollary compares $E_{v}\left( t\right) $\ to the initial energy $%
E_{v}\left( 0\right) .$

\begin{corollary}
\label{coro4}Under the assumptions \emph{(\ref{tlike}) }and \emph{(\ref{ic}),%
} the energy of the solution of Problem \emph{(\ref{waveq})} satisfies%
\begin{multline}
\frac{\min \left\{ 1+v,\left\vert \gamma _{\eta }\right\vert ^{\left(
1+v\right) }\left( 1-v\right) \right\} }{\max \left\{ \left\vert \gamma
_{\eta }\right\vert ^{\left( 1+v\right) }\left( 1+v\right) ,\gamma _{\eta
}^{2}\left( 1-v\right) \right\} }E_{v}\left( 0\right) e^{-\frac{1-v^{2}}{L}%
\ln \left\vert \gamma _{\eta }\right\vert t}\leq E_{v}\left( t\right)
\label{stab0} \\
\leq \frac{\max \left\{ \left\vert \gamma _{\eta }\right\vert ^{\left(
1+v\right) }\left( 1+v\right) ,\gamma _{\eta }^{2}\left( 1-v\right) \right\} 
}{\min \left\{ 1+v,\left\vert \gamma _{\eta }\right\vert ^{\left( 1+v\right)
}\left( 1-v\right) \right\} }E_{v}\left( 0\right) e^{-\frac{1-v^{2}}{L}\ln
\left\vert \gamma _{\eta }\right\vert t},\text{ \ for}\ t\geq 0.
\end{multline}
\end{corollary}

\begin{proof}
Since (\ref{est0}) holds also for $t=0$, then (\ref{stab0}) follows by
combining the two inequalities%
\begin{multline*}
\frac{e^{\frac{1-v^{2}}{L}\ln \left\vert \gamma _{\eta }\right\vert t}}{\max
\left\{ \left\vert \gamma _{\eta }\right\vert ^{\left( 1+v\right) }\left(
1+v\right) ,\gamma _{\eta }^{2}\left( 1-v\right) \right\} }E_{v}\left(
t\right) \leq \frac{2}{L}\sum_{n\in \mathbb{%
%TCIMACRO{\U{2124} }%
%BeginExpansion
\mathbb{Z}
%EndExpansion
}}\left\vert \omega _{n}a_{n}\right\vert ^{2} \\
\leq \frac{1}{\min \left\{ 1+v,\left\vert \gamma _{\eta }\right\vert
^{\left( 1+v\right) }\left( 1-v\right) \right\} }E_{v}\left( 0\right)
\end{multline*}%
and 
\begin{multline*}
\frac{1}{\max \left\{ \left\vert \gamma _{\eta }\right\vert ^{\left(
1+v\right) }\left( 1+v\right) ,\gamma _{\eta }^{2}\left( 1-v\right) \right\} 
}E_{v}\left( 0\right) \leq \frac{2}{L}\sum_{n\in \mathbb{%
%TCIMACRO{\U{2124} }%
%BeginExpansion
\mathbb{Z}
%EndExpansion
}}\left\vert \omega _{n}a_{n}\right\vert ^{2} \\
\leq \frac{e^{\frac{1-v^{2}}{L}\ln \left\vert \gamma _{\eta }\right\vert t}}{%
\min \left\{ 1+v,\left\vert \gamma _{\eta }\right\vert ^{\left( 1+v\right)
}\left( 1-v\right) \right\} }E_{v}\left( t\right) ,
\end{multline*}%
for\ $t\geq 0.$
\end{proof}

The next corollary gives more simple estimates, but less sharper then (\ref%
{est0}) and (\ref{stab0}), for the energy $E_{v}.$

\begin{corollary}
\label{coro5}Under the assumptions \emph{(\ref{tlike}) }and \emph{(\ref{ic}),%
} the energy of the solution of Problem \emph{(\ref{waveq})} satisfies%
\begin{multline}
\left( 1-v\right) \left( \frac{2}{L}\sum_{n\in \mathbb{%
%TCIMACRO{\U{2124} }%
%BeginExpansion
\mathbb{Z}
%EndExpansion
}}\left\vert \omega _{n}a_{n}\right\vert ^{2}\right) \text{ }e^{-\frac{%
1-v^{2}}{L}\ln \left\vert \gamma _{\eta }\right\vert t}\leq E_{v}\left(
t\right)  \label{est1} \\
\leq \gamma _{\eta }^{2}\left( 1+v\right) \left( \frac{2}{L}\sum_{n\in 
\mathbb{%
%TCIMACRO{\U{2124} }%
%BeginExpansion
\mathbb{Z}
%EndExpansion
}}\left\vert \omega _{n}a_{n}\right\vert ^{2}\right) \text{ }e^{-\frac{%
1-v^{2}}{L}\ln \left\vert \gamma _{\eta }\right\vert t},\text{ \ for}\ t\geq
0
\end{multline}%
and%
\begin{equation}
\frac{1}{\gamma _{\eta }^{2}\gamma _{v}}E_{v}\left( 0\right) e^{-\frac{%
1-v^{2}}{L}\ln \left\vert \gamma _{\eta }\right\vert t}\leq E_{v}\left(
t\right) \leq \gamma _{\eta }^{2}\gamma _{v}E_{v}\left( 0\right) e^{-\frac{%
1-v^{2}}{L}\ln \left\vert \gamma _{\eta }\right\vert t},\text{ \ for}\ t\geq
0.  \label{stab1}
\end{equation}
\end{corollary}

\begin{proof}
Since $0\leq v<1$, then it suffices to simplify the constants in (\ref%
{est0}) and (\ref{stab0}) using the fact that $1\leq \left\vert
\gamma _{\eta }\right\vert \leq \left\vert \gamma _{\eta }\right\vert
^{\left( 1+v\right) }<\gamma _{\eta }^{2}$.
\end{proof}

\begin{remark} This case with a dashpoint damping at the inlet bully can be easily investigated by
replacing $v$ by $-v$ in the above sections. In Corollary, \ref{coro3} we
have to change $v$ by $\left\vert v\right\vert ,$ i.e. 
\begin{equation*}
\frac{\mathcal{E}_{v}\left( 0\right) }{1+\left\vert v\right\vert }\leq
E_{v}\left( t\right) \leq \frac{\mathcal{E}_{v}\left( 0\right) }{%
1-\left\vert v\right\vert }\text{ and }\frac{E_{v}\left( 0\right) }{\gamma
_{\left\vert v\right\vert }}\leq E_{v}\left( t\right) \leq \gamma
_{\left\vert v\right\vert }E_{v}\left( 0\right) ,\ \text{for }t\geq 0.
\end{equation*}%
More importantly, we still have the same exponential decay $e^{-\frac{1-v^{2}%
}{L}\ln \left\vert \gamma _{\eta }\right\vert t}$ when $\eta >0.$ The
analogue of estimations (\ref{est1}) is 
\begin{multline*}
\frac{2\left( 1-\left\vert v\right\vert \right) }{L}\sum_{n\in \mathbb{%
%TCIMACRO{\U{2124} }%
%BeginExpansion
\mathbb{Z}
%EndExpansion
}}\left\vert \omega _{n}a_{n}\right\vert ^{2}\text{ }e^{-\frac{1-v^{2}}{L}%
\ln \left\vert \gamma _{\eta }\right\vert t}\leq E_{v}\left( t\right)  \\
\leq \frac{2\gamma _{\eta }^{2}\left( 1+\left\vert v\right\vert \right) }{L}%
\sum_{n\in \mathbb{%
%TCIMACRO{\U{2124} }%
%BeginExpansion
\mathbb{Z}
%EndExpansion
}}\left\vert \omega _{n}a_{n}\right\vert ^{2}\text{ }e^{-\frac{1-v^{2}}{L}%
\ln \left\vert \gamma _{\eta }\right\vert t},\text{ for}\ t\geq 0
\end{multline*}%
and $\gamma _{v}$ is replaced by $\gamma _{\left\vert v\right\vert }$ in (%
\ref{stab1}). \medskip 
\end{remark}

\section*{Acknowledgements} The authors have been supported by the General
Direction of Scientific Research and Technological Development (Algerian
Ministry of Higher Education and Scientific Research) PRFU \#
C00L03UN280120220010.

\subsection*{ORCID} Abdelmouhcene Sengouga https://orcid.org/0000-0003-3183-7973.

\end{document}